\documentclass[12pt]{article}
\usepackage{e-jc}
\usepackage{amsmath, amsthm, amssymb}
\usepackage{tikz}
\usepackage{tkz-graph}
\usepackage{marginnote}
\usepackage{verbatim}
\usepackage{caption}
\usepackage{hyperref}

\usepackage{sectsty}
\allsectionsfont{\sffamily}

\newcommand\dist{{\textrm{dist}}}

\theoremstyle{plain}
\newtheorem{thm}{Theorem}
\newtheorem{prop}[thm]{Proposition}
\newtheorem{lem}[thm]{Lemma}

\newtheorem{obs}[thm]{Observation}

\theoremstyle{definition}

\newtheorem*{keydef}{Key Definition}
\theoremstyle{remark}

\newcommand{\card}[1]{\left|#1\right|}

\def\chiD{{\chi_D}}
\def\chiDl{{\chi_D^{\ell}}}
\def\aut{{\textrm{Aut}}}
\newcommand{\vph}{\varphi}

\newcommand{\aside}[1]{\marginnote{\scriptsize{#1}}[0cm]}

\newcommand{\Emph}[1]{\emph{#1}\aside{#1}}

\dateline{July 17, 2017}{June 19, 2018}{July 13, 2018}
\Copyright{The author. Released under the CC BY license (International 4.0).}
\specs{P3.5}{25(3)}{2018}

\title{Proper Distinguishing Colorings with Few Colors\\ for Graphs with Girth at
Least 5} 
\author{Daniel W. Cranston\thanks{Department of Mathematics and Applied
Mathematics, Virginia Commonwealth University, Richmond, VA;
\texttt{dcranston@vcu.edu}; 
This research is partially supported by NSA Grant 
H98230-16-0351.}
}

\begin{document}

\maketitle
\begin{abstract}
The distinguishing chromatic number, $\chiD(G)$, of a graph $G$ is
the smallest number of colors in a proper coloring, $\vph$, of $G$, such that
the only automorphism of $G$ that preserves all colors of $\vph$ is the
identity map.  Collins and Trenk conjectured that if $G$ is connected with
girth at least 5 and $G\ne C_6$, then $\chiD(G)\le \Delta+1$.  We prove this
conjecture.
\end{abstract}
\bigskip

\section{Introduction}
A \emph{$k$-coloring} $\vph$ of a graph $G$ is a map $\vph:V(G)\to
\{1,\ldots,k\}$.  If $\vph(v)\ne \vph(w)$ whenever $vw\in E(G)$, then $\vph$
is \emph{proper}.  A vertex, $w$, is \Emph{fixed} by a coloring $\vph$ if every
automorphism $f$ that preserves all colors (that is $\vph(f(v))=\vph(v)$ for all
$v\in V(G)$) has $f(w)=w$.
A coloring $\vph$ of a graph $G$ is 
\Emph{distinguishing} if it fixes every vertex; equivalently,
if the only automorphism $f$ of $G$ with $\vph(f(v))=\vph(v)$
for every vertex $v$ is the identity map.  
The \emph{distinguishing chromatic number}, \Emph{$\chiD(G)$}, is the
smallest number of colors in a proper distinguishing coloring.
In a breadth-first (search) spanning
tree, \Emph{level $i$} is the set of vertices at distance $i$ from the root.
We write $\Delta(G)$ for the maximum degree of a graph $G$, and write
$\Delta$ when $G$ is clear from context.

The distinguishing chromatic number was introduced by Collins and
Trenk~\cite{CollinsT06}.  They showed that $\chiD(G)\le 2\Delta$ for every
connected graph $G$, with equality only if $G\in\{K_{\Delta,\Delta},C_6\}$.
Further, they conjectured that no connected graph has $\chiD(G)=2\Delta-1$.
Laflamme and Seyffarth~\cite{LS} confirmed this conjecture for bipartite graphs,
with the exception of $K_{\Delta,\Delta-1}$.
For every tree $T$, Collins and Trenk proved $\chiD(T)\le \Delta+1$.  Collins,
Hovey, and Trenk~\cite{CHT} studied $\chiD(G)$ in terms of $\aut(G)$, the
automorphism group of $G$.
Cavers and Seyffarth~\cite{CS} characterized graphs $G$ with
$\chiD(G)\in\{|V(G)|-1,|V(G)|-2\}$.  Choi, Hartke, and Kaul~\cite{CHK} studied
$\chiD$ for cartesian products of graphs, and Cheng~\cite{cheng} studied it for
interval graphs.

The goal of this note is to prove Theorem~\ref{mainthm}, which was
conjectured\footnote{Through a minor oversight, the original conjecture did not
exclude $C_6$.} by Collins and Trenk~\cite{CollinsT06}, and proved for bipartite
graphs by Alikhani and Soltani~\cite{AS}.

\begin{thm}
If $G$ is connected with girth at least 5 and $G\ne C_6$, then $\chiD(G)\le \Delta+1$.
\label{mainthm}
\end{thm}

Theorem~\ref{mainthm} is best possible in two ways.  First, the star $K_{1,\Delta}$
requires $\Delta+1$ colors.  In fact, Collins and Trenk characterized the trees
$T$ for which $\chiD(T)=\Delta(T)+1$, and there are an infinite number of these
for each $\Delta$.  Thus, we cannot improve the upper bound on $\chiD$, even if
we require larger girth, or impose some requirement on $\Delta$.  Second, for
each $\Delta\ge 3$, they constructed infinitely many graphs $G$ with girth 4,
maximum degree $\Delta$, and $\chiD(G)=2\Delta-2$.  Thus, we cannot relax the
girth bound.

\begin{keydef}
Let $G$ be a connected graph of girth at least 5, and $w\in V(G)$.
Let $T$ be a breadth-first spanning tree, rooted at $w$.  Let $\sigma$ denote
the vertex order of $T$, and for each $v\in V(G)$, let $\sigma_v$ denote the set
of vertices that precede $v$ in $\sigma$.  When $\sigma'$ is a prefix of $\sigma$,  
we often slightly abuse notation by writing $\sigma'$ to mean the set of
vertices in that order.  
Let $\vph'$ be a proper coloring of the
subgraph induced by some $\sigma'$.  We say \Emph{color
greedily}, with respect to $\sigma$ and $\vph'$, to mean that we
extend $\vph'$ to a proper coloring $\vph$ by coloring each $v 
\in \sigma\setminus \sigma'$ in order, subject to the following two constraints.
\textbf{(i)} If $v$ has a neighbor in $\sigma_v$ (other than
its parent), then $\vph(v)$ is the smallest color not used on $N(v)\cap
\sigma_v$. 
\textbf{(ii)} If $v$ has no neighbor in $\sigma_v$ (other than its parent), then
$\vph(v)$ is the smallest color not used on $\sigma_v \cap S_v$,
where $S_v$ consists of the siblings and parent of $v$.
\end{keydef}

We begin with the following easy lemma. Its proof draws on ideas from a similar
result in~\cite{CollinsT06}, where $G$ has no girth constraint.

\begin{lem}
\label{mainlem}
Let $G$ be a connected graph with girth at least 5, and $w\in V(G)$.  Let $T$ be a
breadth-first spanning tree, rooted at $w$, with vertex order $\sigma$.  
Let $\vph$ be a greedy coloring of $G$, with respect to $\sigma$ and a coloring
$\vph'$ of some non-empty prefix $\sigma'$ of $\sigma$.  If $\vph$ fixes each
vertex in $\sigma'$, then $\vph$ is a proper distinguishing coloring.
\end{lem}

\begin{proof}
By construction, $\vph$ is proper, so we need only prove that $\varphi$ fixes
all vertices.  We use induction on $|V(G)|$, at
each step considering the next vertex, $v$, in $\sigma$ that is not known to be
fixed, along with all of its siblings.  By assumption $\vph$ fixes every vertex
in $\sigma'$.  Consider the first vertex, $v$, that is not yet known to be fixed.
By the induction hypothesis, the parent, $x$, of $v$ is fixed.  So, any
automorphism must map $N(x)$ to $N(x)$.  Suppose that some sibling $y$ of $v$
has a neighbor, $z$, that is already colored.  Since $z$ is already colored, it
comes before $v$ and $y$ in $T$.  So, by the induction hypothesis, $z$ is
fixed.  Since $x$ and $z$ are fixed, and $G$ has no 4-cycle, $y$ is also fixed.
 So all siblings of $v$ with a non-parent neighbor already colored are fixed. 
And by construction each sibling of $v$ without such a neighbor already
colored has a color different from $v$ (or is fixed, since it is in $\sigma'$).
  So $v$ and all its siblings are fixed, which completes the proof.
\end{proof}

Lemma~\ref{mainlem} plays a key role in our proof of Theorem~\ref{mainthm}. 
To illustrate how we use it, we first prove a slightly weaker bound.

\begin{prop}
If $G$ is connected with girth at least 5, then $\chi_D(G)\le \Delta+2$.
\label{prop1}
\end{prop}

\begin{proof}
Choose an arbitrary vertex $w$, and let $T$ be a breadth-first spanning tree from
$w$.  Let $\vph'(w)=\Delta+2$, and let $\sigma'=\{w\}$.
Now we color greedily, with respect to $\vph'$ and $\sigma$.
Note that, except on $w$, this coloring only uses colors from
$\{1,\ldots,\Delta+1\}$, since each vertex $v$ must avoid the colors either (i) on
its neighbors, at most $\Delta$, or (ii) on its parent and siblings, at most
$1+(\Delta-1)$.  Since $w$ is the only vertex colored $\Delta+2$, it is fixed.
Now $\vph$ is a proper distinguishing coloring by Lemma~\ref{mainlem}.
Thus, $\chiD(G)\le \Delta+2$.
\end{proof}

\section{Main Result}
Before proving our main result, we need an observation about the largest color 
possibly used in cases (i) and (ii) of a greedy coloring.

\begin{obs}
\label{obs1}
Fix a connected graph $G$ with girth at least 5.
Choose an arbitrary $w\in V(G)$ and breadth-first spanning tree $T$, rooted at
$w$.  For any prefix $\sigma'$ of $\sigma$, proper coloring $\vph'$ of the
subgraph induced by $\sigma'$, and greedy coloring $\vph$, we have the following.
Each vertex in $V(G)\setminus N(w)$ colored by (i) uses a color no larger than
$\Delta+1$ and uses $\Delta+1$ only if all of its neighbors are already
colored, and all use distinct colors.  Each vertex $v\in V(G)\setminus N(w)$
colored by (ii) uses a color no larger than $\Delta$. 
\end{obs}
\begin{proof}
The first statement is obvious.  The second holds since $v$ has at most
$\Delta-2$ siblings.  
\end{proof}

\begin{proof}[Proof of Theorem~\ref{mainthm}]
Assume the theorem is false, and $G$ is a counterexample minimizing
$\card{V(G)}$.

\begin{claim}$G$ has $\Delta\ge 3$.
\label{clm0}
\end{claim}
We assume that $\card{V(G)}\ge 3$, since otherwise the theorem is trivial.
Since $G$ is connected, this implies that $\Delta\ge 2$.  If $G$ is a path,
then we color one end of the path with 1, and the remaining vertices with 2 and
3 alternating.  If $G$ is a cycle, then
we color its vertices (in cyclic order) as $1,2,3,1,2,\ldots$, where the vertices
after the first five alternate colors 3 and 2.  It is easy to check that these
colorings are proper and distinguishing, unless $G=C_6$.

\begin{claim}$G$ is $\Delta$-regular.
\label{clm1}
\end{claim}
Suppose, to the contrary, that $G$ has a vertex, $w$, with degree at most
$\Delta-1$.  Let $T$\aside{$w$, $T$} be a breadth-first spanning tree rooted at
$w$.  Color $w$ with $\Delta+1$, then color $\sigma\setminus \{w\}$ greedily;
call the resulting coloring \Emph{$\vph$}.  Note that no other vertex with
degree at most $\Delta-1$ uses color $\Delta+1$.  Thus, $w$ is fixed by $\vph$.
 Now $\vph$ is a proper distinguishing coloring of $G$ by
Lemma~\ref{mainlem}.  Thus, $\chiD(G)\le \Delta+1$, a contradiction.

\begin{figure}[!t]
\centering
\begin{tikzpicture}[scale = 8.5]
\tikzstyle{VertexStyle} = []
\tikzstyle{EdgeStyle} = [line width=0.8pt]
\tikzstyle{labelingStyle}=[shape = circle, minimum size = 4pt, inner sep = 3pt, outer sep = 2pt]
\tikzstyle{labeledStyle}=[shape = circle, minimum size = 3pt, inner sep = 1pt, outer sep = 1pt, draw]
\tikzstyle{unlabeledStyle}=[shape = circle, minimum size = 2pt, inner sep = 1pt, outer sep = 1pt, draw, fill]
\Vertex[style = labeledStyle, x = 1.000, y = 0.700, L = \footnotesize {$\Delta+1$}]{v0}
\Vertex[style = unlabeledStyle, x = 1.000, y = 0.400, L = \small {}]{v1}
\Vertex[style = labeledStyle, x = 0.700, y = 0.400, L = \footnotesize {$1$}]{v2}
\Vertex[style = labeledStyle, x = 0.400, y = 0.400, L = \footnotesize {$1$}]{v3}
\Vertex[style = unlabeledStyle, x = 1.300, y = 0.400, L = \small {}]{v4}
\Vertex[style = unlabeledStyle, x = 1.600, y = 0.400, L = \small {}]{v5}
\Vertex[style = unlabeledStyle, x = 0.400, y = 0.150, L = \small {}]{v6}
\Vertex[style = unlabeledStyle, x = 0.450, y = 0.150, L = \small {}]{v7}
\Vertex[style = unlabeledStyle, x = 0.500, y = 0.150, L = \small {}]{v8}
\Vertex[style = labeledStyle, x = 0.250, y = 0.150, L = \footnotesize {$\Delta+1$}]{v9}
\Vertex[style = unlabeledStyle, x = 0.950, y = 0.150, L = \small {}]{v10}
\Vertex[style = unlabeledStyle, x = 1.050, y = 0.150, L = \small {}]{v11}
\Vertex[style = unlabeledStyle, x = 1.100, y = 0.150, L = \small {}]{v12}
\Vertex[style = unlabeledStyle, x = 0.900, y = 0.150, L = \small {}]{v13}
\Vertex[style = unlabeledStyle, x = 1.550, y = 0.150, L = \small {}]{v14}
\Vertex[style = unlabeledStyle, x = 1.650, y = 0.150, L = \small {}]{v15}
\Vertex[style = unlabeledStyle, x = 1.700, y = 0.150, L = \small {}]{v16}
\Vertex[style = unlabeledStyle, x = 1.500, y = 0.150, L = \small {}]{v17}
\Vertex[style = unlabeledStyle, x = 1.250, y = 0.150, L = \small {}]{v18}
\Vertex[style = unlabeledStyle, x = 1.350, y = 0.150, L = \small {}]{v19}
\Vertex[style = unlabeledStyle, x = 1.400, y = 0.150, L = \small {}]{v20}
\Vertex[style = unlabeledStyle, x = 1.200, y = 0.150, L = \small {}]{v21}
\Vertex[style = unlabeledStyle, x = 0.650, y = 0.150, L = \small {}]{v22}
\Vertex[style = unlabeledStyle, x = 0.750, y = 0.150, L = \small {}]{v23}
\Vertex[style = unlabeledStyle, x = 0.800, y = 0.150, L = \small {}]{v24}
\Vertex[style = unlabeledStyle, x = 0.600, y = 0.150, L = \small {}]{v25}
\Vertex[style = labelingStyle, x = 1.080, y = 0.700, L = \footnotesize {$w$}]{v26}
\Vertex[style = labelingStyle, x = 0.460, y = 0.400, L = \footnotesize {$x_1$}]{v27}
\Vertex[style = labelingStyle, x = 0.190, y = -0.100, L = \footnotesize {$x_3$}]{v28}
\Vertex[style = unlabeledStyle, x = 0.250, y = -0.100, L = \small {}]{v29}
\Vertex[style = unlabeledStyle, x = 0.150, y = -0.100, L = \small {}]{v30}
\Vertex[style = unlabeledStyle, x = 0.300, y = -0.100, L = \small {}]{v31}
\Vertex[style = unlabeledStyle, x = 0.350, y = -0.100, L = \small {}]{v32}
\Vertex[style = labelingStyle, x = 0.690, y = -0.100, L = \footnotesize {$x$}]{v33}
\Vertex[style = unlabeledStyle, x = 0.650, y = -0.100, L = \small {}]{v34}
\Vertex[style = labelingStyle, x = 0.340, y = 0.150, L = \footnotesize {$x_2$}]{v35}
\Vertex[style = labelingStyle, x = 0.760, y = 0.400, L = \footnotesize {$y_1$}]{v36}
\Edge[](v1)(v0)
\Edge[](v2)(v0)
\Edge[](v3)(v0)
\Edge[](v4)(v0)
\Edge[](v5)(v0)
\Edge[](v6)(v3)
\Edge[](v7)(v3)
\Edge[](v8)(v3)
\Edge[](v9)(v3)
\Edge[](v10)(v1)
\Edge[](v11)(v1)
\Edge[](v12)(v1)
\Edge[](v13)(v1)
\Edge[](v14)(v5)
\Edge[](v15)(v5)
\Edge[](v16)(v5)
\Edge[](v17)(v5)
\Edge[](v18)(v4)
\Edge[](v19)(v4)
\Edge[](v20)(v4)
\Edge[](v21)(v4)
\Edge[](v22)(v2)
\Edge[](v23)(v2)
\Edge[](v24)(v2)
\Edge[](v25)(v2)
\Edge[](v29)(v9)
\Edge[](v30)(v9)
\Edge[](v31)(v9)
\Edge[](v32)(v9)
\Edge[style={bend right}, label = \small {}, labelstyle={auto=right,
fill=none}](v30)(v34)
\end{tikzpicture}
\caption{A partial coloring, constructed in Claim~\ref{clm2}.\label{fig1}}
\end{figure}
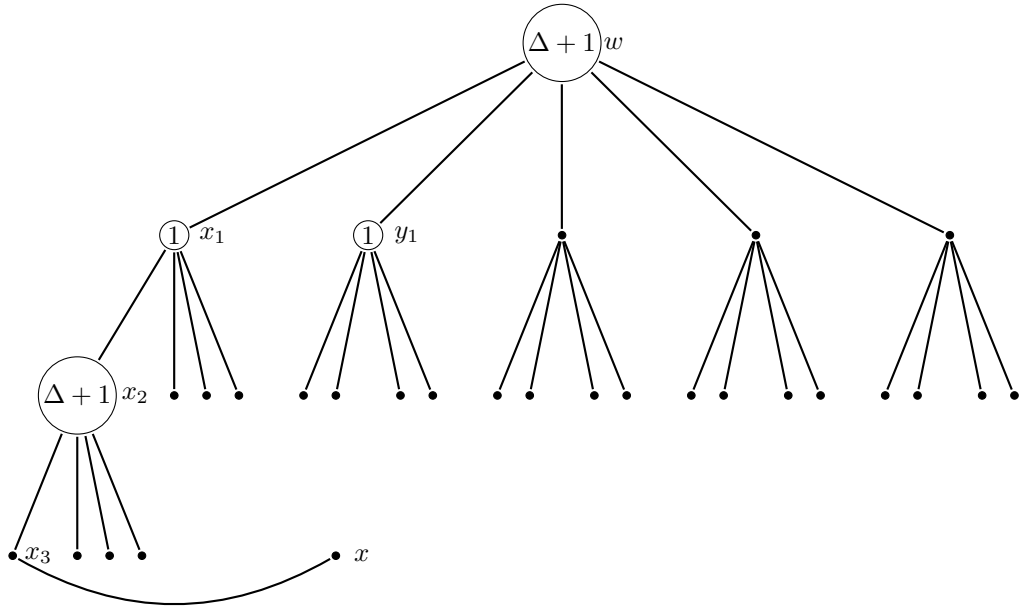

\begin{claim}$G$ has no vertices $w, x, x_3$, with $x$ and $x_3$ adjacent, such
that $\dist(w,x)\ge 3$ and $\dist(w,x_3)=3$.  In particular, $G$ has diameter
at most 3.
\label{clm2}
\end{claim}
If $G$ has a pair of vertices $w$ and $x$ with $\dist(w,x)=4$, then let $x_3$ be
a neighbor of $x$ on a shortest path from $x$ to $w$.  Thus, the second
statement follows from the first.  Now we prove the first.
Assume the contrary.
Let $wx_1x_2x_3x$\aside{$w$, $x_i$, $x$} be a path with $\dist(w,x_i)=i$ for
each $i\in\{1,2,3\}$ and $\dist(w,x)\ge 3$; see Figure~\ref{fig1}.  Let
\Emph{$T$} be a breadth-first spanning tree rooted at $w$, such that each $x_i$
is the first vertex at level $i$ of $T$.  Let \Emph{$y_1$} be the second child
of $w$.  Color $T$ greedily except for the following
modifications.  Color $w$ and $x_2$ with $\Delta+1$, color $x_1$ and $y_1$
with 1, and color $x_3$ after all of its siblings so that the
multiset of colors appearing on $N(x_2)$ differs from that appearing on $N(w)$.
This final step is possible because $x$ is uncolored at the time we color $x_3$,
so we have at least two options for $x_3$.  Call the resulting coloring
\Emph{$\vph$}.

Clearly, $\vph$ is proper and uses at most $\Delta+1$ colors.
We must verify that $\vph$ fixes all vertices.  The proof uses
Lemma~\ref{mainlem}, so we must show that $\vph$ fixes $\sigma'$.  As in
Observation~\ref{obs1}, no vertex other than $w$ and (possibly) $x_2$ uses color
$\Delta+1$ and also has a repeated color in its neighborhood.  By construction,
the multisets of colors used on $N(x_2)$ and $N(w)$ differ.  Thus, $w$ is fixed.
The only children of $w$ that use a common color are $x_1$ and $y_1$; so all of
its other children are fixed.  Consider a child of $y_1$, call it $y_2$.  If
$y_2$ is colored by (ii), then $\vph(y_2)\le \Delta$, by
Observation~\ref{obs1}.  If $y_2$ is colored by (i), then $y_2$ has at most one
colored neighbor (other than $y_1$), since $G$ has girth at least 5.  Thus,
$\vph(y_2)\le 3\le \Delta$.  So $x_1$ has a neighbor (other than $w$) colored
$\Delta+1$, but $y_1$ has no such neighbor.  Hence, $x_1$ and $y_1$ are fixed. 
Let $\sigma'$ be the prefix of $\sigma$ ending with $x_2$ and $\vph'$ the
restriction of $\vph$ to the subgraph induced by $\sigma'$.  Now
Lemma~\ref{mainlem} implies that all vertices are fixed, so $\vph$ is a proper
distinguishing coloring using at most $\Delta+1$ colors, a contradiction.

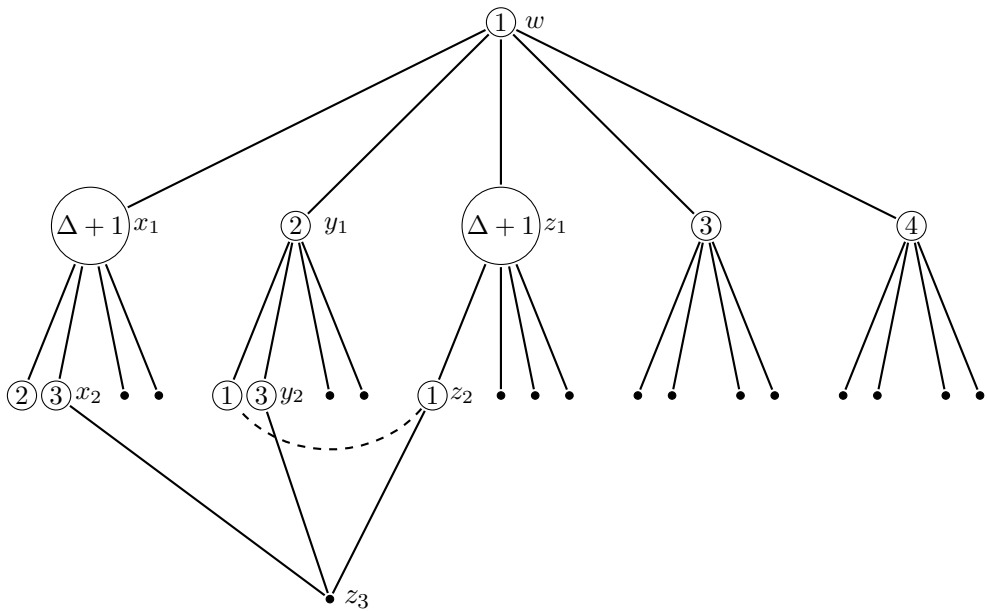
\begin{figure}[!b]
\centering
\begin{tikzpicture}[scale = 9]
\tikzstyle{VertexStyle} = []
\tikzstyle{EdgeStyle} = [line width=.8pt]
\tikzstyle{labelingStyle}=[shape = circle, minimum size = 4pt, inner sep = 3pt, outer sep = 2pt]
\tikzstyle{labeledStyle}=[shape = circle, minimum size = 3pt, inner sep = 1pt, outer sep = 1pt, draw]
\tikzstyle{unlabeledStyle}=[shape = circle, minimum size = 2pt, inner sep = 1pt, outer sep = 1pt, draw, fill]
\Vertex[style = labeledStyle, x = 0.850, y = 0.900, L = \footnotesize {$1$}]{v0}
\Vertex[style = labeledStyle, x = 0.850, y = 0.600, L = \footnotesize {$\Delta+1$}]{v1}
\Vertex[style = labeledStyle, x = 0.550, y = 0.600, L = \footnotesize {$2$}]{v2}
\Vertex[style = labeledStyle, x = 0.250, y = 0.600, L = \footnotesize {$\Delta+1$}]{v3}
\Vertex[style = labeledStyle, x = 1.150, y = 0.600, L = \footnotesize {$3$}]{v4}
\Vertex[style = labeledStyle, x = 1.450, y = 0.600, L = \footnotesize {$4$}]{v5}
\Vertex[style = labeledStyle, x = 0.200, y = 0.350, L = \footnotesize {$3$}]{v6}
\Vertex[style = unlabeledStyle, x = 0.300, y = 0.350, L = \small {}]{v7}
\Vertex[style = unlabeledStyle, x = 0.350, y = 0.350, L = \small {}]{v8}
\Vertex[style = labeledStyle, x = 0.150, y = 0.350, L = \footnotesize {$2$}]{v9}
\Vertex[style = unlabeledStyle, x = 0.850, y = 0.350, L = \small {}]{v10}
\Vertex[style = unlabeledStyle, x = 0.900, y = 0.350, L = \small {}]{v11}
\Vertex[style = unlabeledStyle, x = 0.950, y = 0.350, L = \small {}]{v12}
\Vertex[style = labeledStyle, x = 0.750, y = 0.350, L = \footnotesize {$1$}]{v13}
\Vertex[style = unlabeledStyle, x = 1.400, y = 0.350, L = \small {}]{v14}
\Vertex[style = unlabeledStyle, x = 1.500, y = 0.350, L = \small {}]{v15}
\Vertex[style = unlabeledStyle, x = 1.550, y = 0.350, L = \small {}]{v16}
\Vertex[style = unlabeledStyle, x = 1.350, y = 0.350, L = \small {}]{v17}
\Vertex[style = unlabeledStyle, x = 1.100, y = 0.350, L = \small {}]{v18}
\Vertex[style = unlabeledStyle, x = 1.200, y = 0.350, L = \small {}]{v19}
\Vertex[style = unlabeledStyle, x = 1.250, y = 0.350, L = \small {}]{v20}
\Vertex[style = unlabeledStyle, x = 1.050, y = 0.350, L = \small {}]{v21}
\Vertex[style = labeledStyle, x = 0.500, y = 0.350, L = \footnotesize {$3$}]{v22}
\Vertex[style = unlabeledStyle, x = 0.600, y = 0.350, L = \small {}]{v23}
\Vertex[style = unlabeledStyle, x = 0.650, y = 0.350, L = \small {}]{v24}
\Vertex[style = labeledStyle, x = 0.450, y = 0.350, L = \footnotesize {$1$}]{v25}
\Vertex[style = labelingStyle, x = 0.900, y = 0.900, L = \footnotesize {$w$}]{v26}
\Vertex[style = labelingStyle, x = 0.333, y = 0.600, L = \footnotesize {$x_1$}]{v27}
\Vertex[style = labelingStyle, x = 0.610, y = 0.600, L = \footnotesize {$y_1$}]{v28}
\Vertex[style = labelingStyle, x = 0.930, y = 0.600, L = \footnotesize {$z_1$}]{v29}
\Vertex[style = labelingStyle, x = 0.248, y = 0.350, L = \footnotesize {$x_2$}]{v30}
\Vertex[style = labelingStyle, x = 0.545, y = 0.350, L = \footnotesize {$y_2$}]{v31}
\Vertex[style = labelingStyle, x = 0.793, y = 0.350, L = \footnotesize {$z_2$}]{v32}
\Vertex[style = unlabeledStyle, x = 0.600, y = 0.050, L = \small {}]{v33}
\Vertex[style = labelingStyle, x = 0.640, y = 0.050, L = \small {$z_3$}]{v34}
\Edge[](v1)(v0)
\Edge[](v2)(v0)
\Edge[](v3)(v0)
\Edge[](v4)(v0)
\Edge[](v5)(v0)
\Edge[](v6)(v3)
\Edge[](v7)(v3)
\Edge[](v8)(v3)
\Edge[](v9)(v3)
\Edge[](v22)(v2)
\Edge[](v23)(v2)
\Edge[](v24)(v2)
\Edge[](v25)(v2)
\Edge[](v10)(v1)
\Edge[](v11)(v1)
\Edge[](v12)(v1)
\Edge[](v13)(v1)
\Edge[](v18)(v4)
\Edge[](v19)(v4)
\Edge[](v20)(v4)
\Edge[](v21)(v4)
\Edge[](v14)(v5)
\Edge[](v15)(v5)
\Edge[](v16)(v5)
\Edge[](v17)(v5)
\Edge[](v33)(v6)
\Edge[](v33)(v22)
\Edge[](v33)(v13)
\Edge[style = {bend left=50, dashed}](v13)(v25)
\end{tikzpicture}
\caption{A partial coloring, constructed in Claim~\ref{clm3}.\label{fig2}}
\end{figure}

\begin{claim}$G$ has diameter 2 or $\Delta=3$.
\label{clm3}
\end{claim}
Suppose, to the contrary, that $\Delta\ge 4$ and $G$ has diameter at least 3. 
Claim~\ref{clm2} implies that $G$ has diameter at most 3, so assume $G$ has
diameter 3.  Choose $w,z_1,z_2,z_3\in V(G)$\aside{$w$, $z_i$, $T$} such
that $wz_1z_2z_3$ is a path and $\dist(w,z_i)=i$ for each $i\in \{1,2,3\}$.  
Let \emph{$T$} be a breadth-first spanning tree rooted at $w$.  Each $z_i$ is
at level $i$ in $T$.  Further, by Claim~\ref{clm2}, every neighbor of $z_3$ is
at level 2 in $T$.
Choose $x_1,x_2,y_1,y_2\in V(G)$\aside{$x_i$, $y_i$} such that $wx_1x_2z_3$ and
$wy_1y_2z_3$ are
paths, and $x_1,x_2,y_1,y_2,z_1,z_2$ are distinct; see Figure~\ref{fig2}.
After possibly reordering some siblings, we can assume that $x_1$, $y_1$, and
$z_1$ are (respectively) the first, second, and third children of $w$.  We can
also assume that $z_2$ is the first child of $z_1$ and that $z_3$ is the last
child of $z_2$ (even though $z_3$ would naturally be a child of $x_2$). 
Further, we assume $x_2$
and $y_2$ are the second children, respectively, of $x_1$ and $y_1$.  Finally,
by reordering children of $y_1$ if needed, we assume that $z_2$ is not
adjacent to the first child of $y_1$.  Since $\Delta\ge 4$ (and $G$ has girth at
least 5), $y_1$ has at least 3 children; one of these is $y_2$ and at most one
is adjacent to $z_2$, since $G$ has no 4-cycle.  Similarly, we assume that $y_2$
is not adjacent to the first child of $x_1$.

Begin by coloring $w$ with $1$ and coloring $x_1$ and $z_1$ with $\Delta+1$.
We greedily color the rest of $T$ except for the following modifications;
none of these change the order of the vertices, only our choice of color for a
vertex when we reach it in the order.
No child of $x_1$ uses color 1 (so its children use colors $2,\ldots,\Delta$).
Vertex $z_2$ uses 1 and no other child of $z_1$ uses 1.  (If such a child,
say \Emph{$z_2'$}, has no colored neighbor other than $z_1$, then this happens
naturally as a result of a greedy coloring, since $z_2$ uses 1.  If $z_2'$ does
have a colored neighbor other than $z_1$, then it has at most two: one child of
$x_1$ and one child of $y_1$, since $G$ has no 4-cycles.  Since $\Delta\ge 4$,
$z_2'$ has some available color from among $\{2,3,4\}$ that is not already used
on any of its neighbors, so $z_2'$ can avoid 1;
recall that $z_1$ uses $\Delta+1>4$, so $z_1$ does not forbid a color in
$\{2,3,4\}$.) Note that $x_2$ and $y_2$ both use 3.  When we color $z_3$ (as the
last child of $z_2$), we do so to ensure that
the multisets of colors used on $N(w)$ and $N(z_2)$ differ; this is possible
because $x_2$ and $y_2$ use the same color, so $z_3$ has at least two available
colors. Further, if possible, we also require that $z_3$ not use color
$\Delta+1$; so $z_3$ uses $\Delta+1$ only if all of its neighbors use distinct
colors, except for $x_2$ and $y_2$, which both use 3.  Call the resulting
coloring \Emph{$\vph$}.  

Now we show that $\vph$ fixes all vertices.
Since $\vph(z_2)=1$, vertex $z_1$ is the unique vertex colored $\Delta+1$
with at least two neighbors colored 1, by Observation~\ref{obs1}.  Thus $z_1$
is fixed.  By design, no
child of $z_1$ other than $z_2$ uses color 1.  By our choice of color for $z_3$,
the multisets of colors on $N(w)$ and $N(z_2)$ differ.  Thus, $w$ is fixed and
$z_2$ is fixed.  Since we already know $z_1$ is fixed, each neighbor of $w$ is
fixed.  Now, as in the proof of Lemma~\ref{mainlem}, by induction, we show for
each vertex in $N(w)$ that its children are fixed.  Finally, consider a vertex,
$u$, at level 3 of $T$.  By Claim~\ref{clm2}, each neighbor of $u$ is at level
2.  Thus, since $G$ has no 4-cycle and $u$ has at least 2 fixed neighbors, $u$
is fixed.  So $\chiD(G)\le \Delta+1$, a contradiction.

\begin{claim} $G$ has $\Delta=3$.
\label{clm4}
\end{claim}
Suppose that $\Delta \ge 4$.  By Claim~\ref{clm3}, $G$ has diameter 2.
Since $G$ is $\Delta$-regular with diameter 2 and girth 5, we know that
$\card{V(G)}=\Delta^2+1$.  Choose an arbitrary vertex $w$\aside{$w$, $G'$}, and
let $G'=G\setminus\{w\cup N(w)\}$.  Note that $G'$ is regular, with
degree $\Delta(G)-1$.  
Since $G$ has diameter 2, for each $v_1,v_2\in V(G')$, if $v_1$ and $v_2$ have
no common neighbor (in $G$) among $N(w)$, then $\dist_{G'}(v_1,v_2)\le
2$.  This implies that $G'$ is connected.  
Since $\Delta(G)-1\ge 3$, by the minimality of $G$ we know that $\chiD(G')\le
\Delta(G')+1=\Delta(G)$; let \Emph{$\vph'$} be a coloring of $G'$
showing this.  To extend $\vph'$ to $G$, use color $\Delta(G)+1$ on each
vertex in $N(w)$ and use an arbitrary color (other than $\Delta(G)+1$) on $w$;
we call this coloring \Emph{$\vph$}.  Clearly $w$ is fixed by $\vph$, since
each other vertex has at most one neighbor colored $\Delta(G)+1$.  Further,
each vertex of $G'$ is fixed by $\vph'$, so we get that each vertex of $G'$ is
fixed by $\vph$.  Finally, each vertex $v\in N(w)$ has all of its neighbors
fixed by $\vph$.  Since $G$ has no 4-cycles, $v$ is also fixed by $\vph$. 
Thus, $\chiD(G)\le \Delta(G)+1$, a contradiction.

\begin{claim} $G$ is either bipartite or vertex-transitive.
\label{clm5}
\end{claim}
Suppose there exist vertices \Emph{$w,x_1,y_1$} with $x_1,y_1\in N(w)$ and no
automorphism maps $x_1$ to $y_1$.  Let \Emph{$T$} be a breadth-first spanning
tree, rooted at $w$, with $x_1$ and $y_1$ as the first children of $w$ in $T$.
Color $w$ with $\Delta+1$, color $x_1$ and $y_1$ with 1, and color the rest of
$T$ greedily.  Now $w$ is fixed, since it is the only vertex colored $\Delta+1$
with two neighbors colored 1.  Vertices $x_1$ and $y_1$ are fixed, since no
automorphism maps one to the other, and all other neighbors of $w$ are fixed,
since they receive distinct colors.  All remaining
vertices are fixed by Lemma~\ref{mainlem}.
So $\chiD(G)\le \Delta+1$, a contradiction.

Instead, assume that for every vertex $w$ and $x_1,y_1\in N(w)$\aside{$w$,
$x_1$, $y_1$} some
automorphism maps $x_1$ to $y_1$.  Thus, for each pair of vertices
$x_1,y_1$ joined by a walk of even length, some automorphism maps $x_1$ to
$y_1$ (the proof is by induction on the length of the walk).  If $G$ is
bipartite, we are done; so assume it is not.  Let \Emph{$C$} be an odd cycle
in $G$.  Since $G$ is connected, for every pair of vertices $x_1,y_1\in V(G)$,
there exists an $x_1,y_1$-walk of even length.  Consider a walk from $x_1$ to
$y_1$ that visits $C$.  If the walk has odd length, then we extend it by going
around $C$ once.  Thus, $G$ is vertex-transitive, as desired.

\begin{claim}
$G$ is either the Petersen graph or the Heawood graph; see Figure~\ref{fig3}.
\label{clm6}
\end{claim}
Recall that $G$ is 3-regular, by Claims~\ref{clm1} and~\ref{clm4}.
Choose an arbitrary vertex $w$, and let $T$\aside{$w$, $T$} be a breadth-first
spanning tree rooted at $w$.  Since $G$ has girth at least 5, the number of
vertices at level 2 is $3(3-1)=6$.  By Claim~\ref{clm2}, $G$ has diameter at
most 3, so any additional vertices of $G$ are at level 3.  Further, by
Claim~\ref{clm2}, no pair of vertices at level
3 are adjacent.  So each vertex at level 3 has all three of its neighbors at level 2.
Since each of the 6 vertices at level 2 has at most two neighbors at level 3, the
number of vertices at level 3 is at most $\frac23(6)=4$.  Thus, $\card{V(G)}\le
1+3+6+4=14$.  
If $G$ has girth at least 6, then $G$ must be the Heawood graph, since it is
the unique 3-regular graph with girth at least 6 and at most 14 vertices.
Otherwise, $G$ has a 5-cycle.  So Claim~\ref{clm5} implies that $G$ is
vertex-transitive.  

To finish, we simply check that the only 3-regular vertex-transitive graphs
with girth at least 5 and at most 14 vertices are the Petersen graph and the
Heawood graph.  A convenient reference to verify this is~\cite{mckay79}, which
catalogues all vertex-transitive graphs on at most 19 vertices.

\begin{figure}[!b]
\centering
\begin{tikzpicture}[line width=0.8pt]
\tikzstyle{blavert}=[circle, draw, fill=white, inner sep=0pt, minimum width=11pt]

\begin{scope}[thick, scale=2, rotate = -30, xscale=-1]
\draw \foreach \i/\lab in {1/2, 2/1, 3/2, 4/4, 5/3, 6/2, 7/4, 8/2, 9/3}
{
(\i*360/9:1) node[blavert](v\i){\footnotesize $\lab$} -- (\i*360/9+360/9:1)
};
\draw node[blavert](v10){\footnotesize $1$};

\Edge[style={bend right}, labelstyle={auto=right, fill=none}](v1)(v5)
\Edge[style={bend right}, label = \small {}, labelstyle={auto=right, fill=none}](v4)(v8)
\Edge[style={bend right}, label = \small {}, labelstyle={auto=right, fill=none}](v7)(v2)
\Edge[](v3)(v10)
\Edge[](v6)(v10)
\Edge[](v9)(v10)
\end{scope}

\begin{scope}[xshift = 6.5cm, thick,scale=2, rotate=-3*360/14, xscale=-1]
\draw \foreach \i in {2, 4, ..., 14}
{
(\i*360/14:1) node[blavert]{} -- (\i*360/14+360/14:1)
(\i*360/14+360/14:1) node[blavert]{} -- (\i*360/14+720/14:1)
(\i*360/14:1)  -- (\i*360/14+5*360/14:1)
};
\draw \foreach \i/\lab in 
{1/2, 2/3, 3/2, 4/3, 5/2, 6/1, 7/2, 8/1, 9/4, 10/3, 11/2, 12/1, 13/2, 14/4}
{
(\i*360/14:1) node[]{\footnotesize $\lab$} 
};

\end{scope}
\end{tikzpicture}
\caption{Proper distinguishing 4-colorings of the Petersen graph and Heawood
graph.\label{fig3}}
\end{figure}
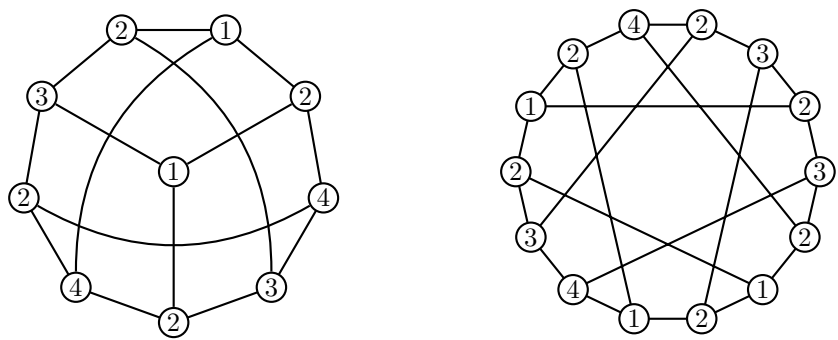

\begin{claim}
$G$ does not exist.
\label{clm7}
\end{claim}
Consider the 4-colorings in Figure~\ref{fig3} of the Petersen graph (left) and
Heawood graph (right).  Clearly, both colorings are proper.  In the coloring of
the Petersen graph, every two vertices with the same color have distinct
multisets of colors on their neighborhoods.  Thus, the coloring is
distinguishing.

Now consider the coloring of the Heawood graph.  The vertices colored 4 are
fixed, since the multisets on their neighborhoods differ.  The same is true for
the vertices colored 2, except for two with the multiset $\{1,1,3\}$.  But for
exactly one of these vertices the neighbor colored 3 has a neighbor colored 4.
Thus, the vertices colored 2 are fixed.  Now every other vertex has at least two
neighbors fixed.  Since $G$ has no 4-cycles, every vertex is fixed.  So the
coloring is distinguishing.  This proves claim~\ref{clm7}, which finishes
the proof of Theorem~\ref{mainthm}.
\end{proof}

We conclude with a few remarks.
The \emph{distinguishing list chromatic number}, \Emph{$\chiDl(G)$}, is the
smallest $k$ such that given any list assignment $L$ with $\card{L(v)}\ge k$
for all $v\in V(G)$, there exists a proper distinguishing coloring $\vph$ of
$G$ such that $\vph(v)\in L(v)$ for all $v\in V(G)$.
In general, $\chiDl(G)$ can be much larger than $\chiD(G)$.  For example, form a
graph $G$ from $K_{1000,1000}$ by adding a pendant edge at each vertex and
subdividing these edges so the resulting paths have distinct lengths.  The only
automorphism of $G$ is the identity, so $\chiD(G)=\chi(G)=2$.  However,
$\chiDl(G)$ is at least the list-chromatic number, which is greater than 2 (and
grows arbitrarily large on complete bipartite graphs).  In contrast, we note
that the bound on $\chiD$ in Proposition~\ref{prop1} also holds for $\chiDl$.

\begin{prop}
If $G$ is connected with girth at least 5, then $\chiDl(G)\le \Delta+2$.
\label{prop2}
\end{prop}
\begin{proof}
The proof is nearly identical to that of Proposition~\ref{prop1}.  
We choose an arbitrary vertex $w$, choose $\alpha\in L(w)$, color $w$ with
$\alpha$ and let $L'(v)=L(v)-\alpha$ for every other vertex $v$.
Now find a breadth-first spanning tree, $T$, rooted at $w$, and color greedily
from $L'$ with respect to $\sigma$.  It is straightforward to check that the
proof of Lemma~\ref{mainlem} still holds in this more general setting.
\end{proof}

\begin{obs}
All proofs in this paper are constructive, and imply algorithms to find the
colorings, and these algorithms can be easily implemented to run in
polynomial time.
\end{obs}
\begin{proof}
We can easily check that a graph is regular, and compute a breadth-first
spanning tree.  To determine diameter, we use an algorithm for all-pairs
shortest path.  In general, computing the automorphisms of a graph can be
difficult.
However, we only need to do this for graphs on at most 14 vertices, so we can
finish in constant time, simply considering every possible map from $V(G)$ to
$V(G)$, and checking whether it is an automorphism.
\end{proof}

\begin{question}
When $G$ is connected with girth at least 5, is $\chiDl(G)\le \Delta+1$?
\label{ques1}
\end{question}

Two extensions of list chromatic number are online list chromatic number,
denoted
$\chi^{OL}$, and correspondence chromatic number, denoted $\chi^{corr}$;
see~\cite{Zhu-paint} and~\cite{DP-corr} for the definitions.  When we require
also that our colorings be distinguishing, we get the parameters $\chi^{OL}_D$
and $\chi^{corr}_D$.  It is natural to study the maximum values of these
parameters on certains classes of graphs.  In particular, what are the maximum
values on connected graphs with girth at least 5?  Does the bound of
Proposition~\ref{prop2}, or even Question~\ref{ques1}, still hold?

\subsection*{Acknowledgments}
The author is grateful for his support by NSA grant H98230-16-0351.
He also thanks Howard Community College for its hospitality during the
preparation of this paper.

\bibliographystyle{plain}
\bibliography{GraphColoring1}
\end{document}